\documentclass[11pt,A4paper]{article}
\usepackage{amsfonts}
\usepackage[margin=2cm]{geometry}
\usepackage{graphicx,color}
\usepackage[hang,small,bf]{caption}
\usepackage{xspace}
\usepackage{enumerate}
\usepackage{amssymb,amsmath}
\usepackage{amsthm}
\usepackage{subfig}

\usepackage[T1]{fontenc}
\usepackage[utf8]{inputenc}
\usepackage{authblk}
\usepackage{hyperref}
\usepackage{apacite}

\newtheorem{theorem}{Theorem}

\newtheorem{corollary}[theorem]{Corollary}
\newtheorem{proposition}[theorem]{Proposition}
\newtheorem{lemma}[theorem]{Lemma}

\newtheorem{definition}[theorem]{Definition}

\begin{document}

\author[1]{Spyros Angelopoulos\footnote{Corresponding author.}}
\author[2]{Thomas Lidbetter}
\affil[1]{CNRS and Sorbonne   Universit\'e,  Laboratoire   d’Informatique  de  Paris   6, 4 Place Jussieu, Paris, France 75252.
{\tt spyros.angelopoulos@lip6.fr}}
\affil[2]{Department of Management Science and Information Systems, Rutgers Business School, Newark, NJ 07102, USA.
{\tt tlidbetter@business.rutgers.edu}}

\title{Competitive Search in a Network}

\maketitle

\begin{abstract}
We study the classic problem in which a {\em Searcher} must locate a hidden point, also called the {\em Hider} in a network,
starting from a root point. The network may be either bounded or unbounded, thus generalizing well-known settings such as 
linear and star search. We distinguish between pathwise search, in which the Searcher follows
a continuous unit-speed path until the Hider is reached, and expanding search, in which, at any point in time, 
the Searcher may restart from any previously reached point. The former has been the usual paradigm for studying search games, whereas the latter is a more recent paradigm that can model real-life settings such as hunting for a fugitive, demining a field, or search-and-rescue operations. We seek both deterministic and randomized search strategies that minimize the {\em competitive ratio}, namely the worst-case ratio of the Hider's discovery time, divided by the shortest path to it from the root. Concerning expanding search, we show that a simple search strategy that applies a ``waterfilling'' principle has optimal deterministic competitive ratio; in contrast, we show that the optimal randomized competitive ratio is attained by fairly complex strategies
even in a very simple network of three arcs. Motivated by this observation, we present and analyze an expanding search strategy that is a 5/4 approximation of the randomized competitive ratio. Our approach is also applicable to pathwise search, for which we give a strategy that is
a 5 approximation of the randomized competitive ratio, and which improves upon strategies derived from previous work.
\end{abstract}

\noindent
{\bf Keywords:} Game Theory; Search games; Competitive analysis; Networks.

\section{Introduction}
\label{sec:intro}

We consider the classic setting in which a mobile {\em Searcher} must locate a stationary hidden object, called the {\em Hider}, 
in a network $Q$ with given arc lengths. This general problem goes back to early work in~\cite{isaacs1999differential} and~\cite{gal1979search}, 
who introduced it in the context of the standard, {\em pathwise} search; namely, in this usual setting, the Searcher moves at unit speed starting from a given point $O$ of the network that we call the {\em root}, and the {\em search time} is defined as the first time at which the Searcher reaches the Hider. 
A different approach was recently introduced in~\cite{AL:expanding}, and allows the Searcher to move at infinite speed
within any region of the  network that it has already visited; see Section~\ref{subsec:prelim.expanding} for a formal definition. 
This paradigm captures several situations in which the cost of re-exploration is negligible, compared to the cost of first-time exploration, 
and thus can model settings such as mining for coal, hunting a fugitive, or searching for a missing person. 

The above works take the approach of seeking mixed, i.e., randomized search strategies, with the objective of minimizing the expected search time,
in the worst case; that is, the maximum expected search time over all hiding points in the network. This is accomplished by studying a zero-sum game
with payoff the search time, between a minimizing Searcher and a maximizing Hider. In this paper, instead, we study a normalized
variant of the search time, in which the search time for reaching a point $p$ in $Q$ is divided by the {\em shortest path} from $O$ to $p$ in $Q$; we call
this the {\em normalized search time of $p$}.  The objective thus becomes to find strategies that minimize the worst-case (normalized) search time, by considering all points in the network $Q$. 

This normalized formulation was first applied in search games over unbounded domains, such as the {\em linear search}~\cite{beck:yet.more}
and {\em star search}~\cite{gal:general} problems. Normalization is essential in unbounded domains, since otherwise the Hider can induce unbounded search times,
by hiding arbitrarily far from $O$. Further motivation behind the study of normalized objectives is provided by 
{\em competitive analysis of online algorithms} in which the algorithm operates in a status of total uncertainty about the input, and the normalized
objective describes how much close the algorithm's output is, in comparison to an ideal solution with complete information on the input. 
For this reason,~\cite{jaillet:online} refer to searching under the competitive ratio as {\em online searching}. 
Competitive analysis has been applied even in search games over a bounded domain, as in~\cite{koutsoupias:fixed,fleischer:online,ANGELOPOULOS201951}. We will refer to the {\em competitive ratio of a strategy} as the worst-case normalized search time 
among all points of the network\footnote{In~\cite{koutsoupias:fixed,ANGELOPOULOS201951} the term {\em search ratio} is used in order to refer to the 
competitive ratio. In this work we choose the latter, since it is more prevalent, and since it has been adopted both by the Operations Research
and the Computer Science communities; see e.g., the discussion in~\cite{searchgames} and~\cite{jaillet:online}.}. 
Lastly, we define the {\em competitive ratio of a network $Q$} (with a given root $O$) 
as the minimum competitive ratio of any search strategy for $Q$. We will further distinguish between the {\em deterministic} and the {\em randomized} competitive ratios, depending on whether we consider deterministic or randomized search strategies, respectively. 

\subsection{Main results}
\label{subsec:contribution}

In this work we study the competitive ratio of general networks, both in the expanding and the pathwise search paradigms, which are defined precisely in Section~\ref{sec:preliminaries}. For expanding search, we first show in Section~\ref{sec:deterministic} that the deterministic competitive ratio is achieved by a simple strategy. 
This strategy can be visualized as the frontier that is obtained  by ``flooding'' the network starting at $O$, assuming that the arcs
represent pipes of corresponding lengths. We then move to randomized strategies for expanding search in Section~\ref{sec:random}. Here, we show that, unlike the deterministic case, optimal search strategies have a complex statement even on a very simple network that consists of three arcs. Motivated by this observation, we give approximations to the value
of the game. First, we show that the randomized competitive ratio of a network is within a factor of 2 of its deterministic competitive ratio, and this bound is tight. 
More importantly, we give a class of randomized strategies that approximate the randomized competitive ratio of a network within a factor of $5/4$. 
This class of strategies is based on iterative applications of Randomized Depth-First-Search, in randomly chosen and increasingly large subsets of the network. 
This strategy is inspired by a randomized strategy used for tree graphs in the {\em discrete} setting, namely when the Hider can only hide over 
vertices of a given, finite tree~\cite{ANGELOPOULOS201951}. We emphasize that, unlike~\cite{ANGELOPOULOS201951}, in this work, the search domain 
may be substantially more complex than a tree, and it may also be unbounded. 

Moreover, we give further approximations of the value of the game by relating the payoff of the search strategy to the function $f_Q$, which 
informally gives the measure of the set of points within a certain given radius from the root. As a corollary, we show that if the function $f_Q$ is concave, the randomized competitive ratio is identical to the deterministic one. This finding may have practical implications in the context of searching in a big city, since the road network is naturally much more dense in its center than it its outskirts, and one expects this density to decrease the further we move from the city 
center. 

Our approach in studying expanding search, and more specifically, our lower bounds on the randomized competitive ratio, have implications for
pathwise search as well. More precisely, in Section~\ref{sec:pathwise} we give a randomized pathwise search strategy, inspired by the one for expanding search, which is a 5-approximation
of the randomized competitive ratio. This is an improvement over the $3+2\sqrt{2} \approx 5.828$-approximation that can be derived from techniques
in~\cite{koutsoupias:fixed}. 

In Section~\ref{sec:implementation} we discuss some technicalities relating to the implementation of our search strategies, and in Section~\ref{sec:conclusion} we conclude with directions for future work.

To illustrate the significance of the results and the approaches, consider the star-search problem in which the search domain consists of $m$ infinite,
concurrent rays (Figure~\ref{fig:star.example}). 
Star search has a long history of research, and several of variants of this problem have been studied under the competitive ratio 
(see Chapters 7 and 9 in~\cite{searchgames}). 
It is known that the deterministic competitive ratio is equal to $1+(2\frac{m}{m-1})^\frac{m}{m-1}$~\cite{gal:general}.
In contrast, the randomized competitive ratio is not known (in~\cite{hybrid} optimality is shown under the fairly restrictive assumption of {\em periodic}
strategies). The strategy we obtain in this work has randomized competitive ratio which is at most a factor of 5 from the optimal one. 
Furthermore, the result applies to much more complicated unbounded domains, for instance such as the one depicted in Figure~\ref{fig:general.example}, under the mild
(and necessary) assumption that for any $r>0$, the number of points at distance $r$ from the root of the network is bounded. 
\begin{figure}[htb!]
   \subfloat[A star domain in which $m=4$.]
   {
   \begin{minipage}{0.45\textwidth}
    \centering
  \includegraphics[scale=0.25]{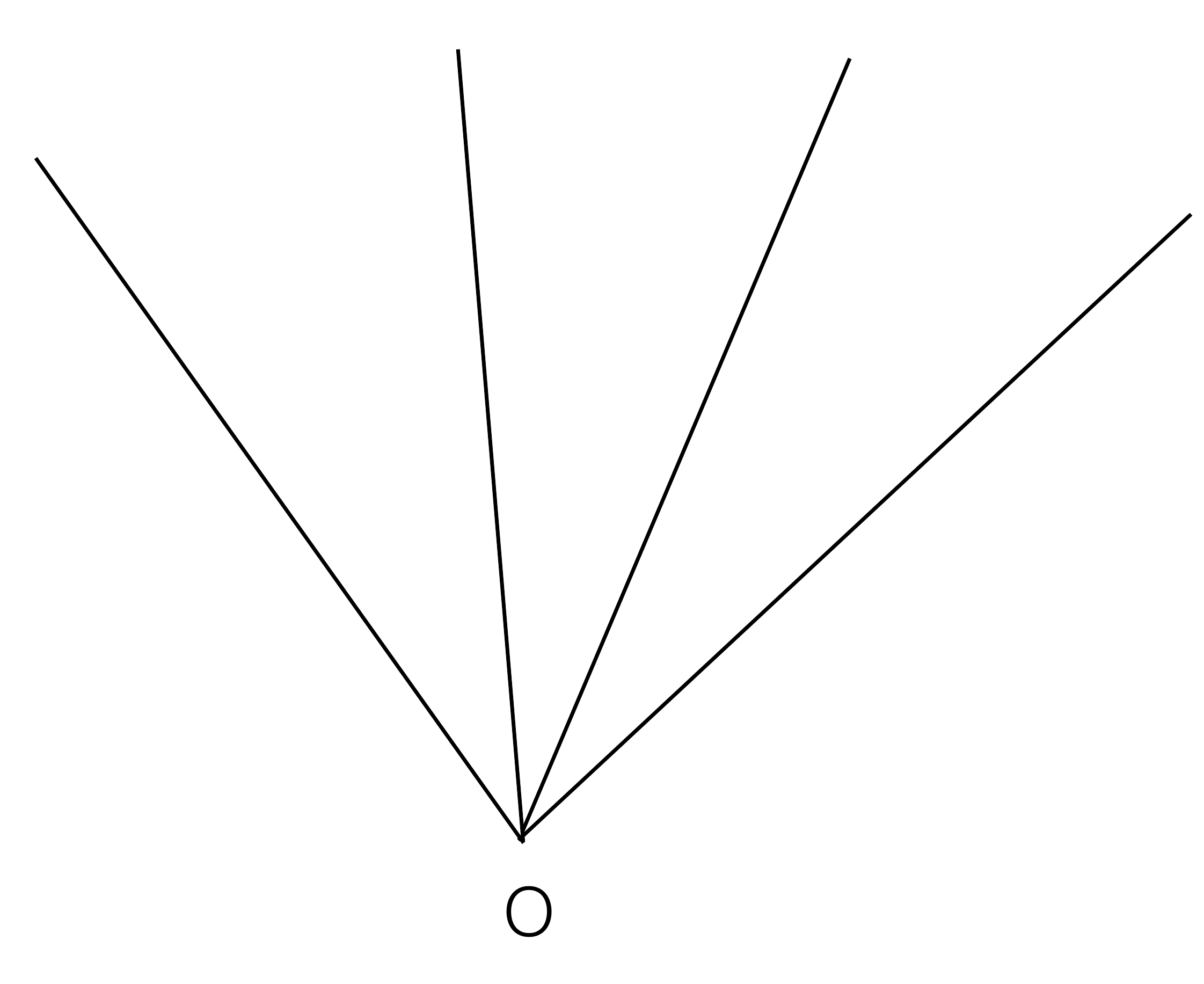} 
    \end{minipage}
    \label{fig:star.example}
    } \qquad
  \subfloat[An example of a search domain studied in this work.]
  {
    \begin{minipage}{0.45\textwidth}
    \centering
       \includegraphics[scale=0.25]{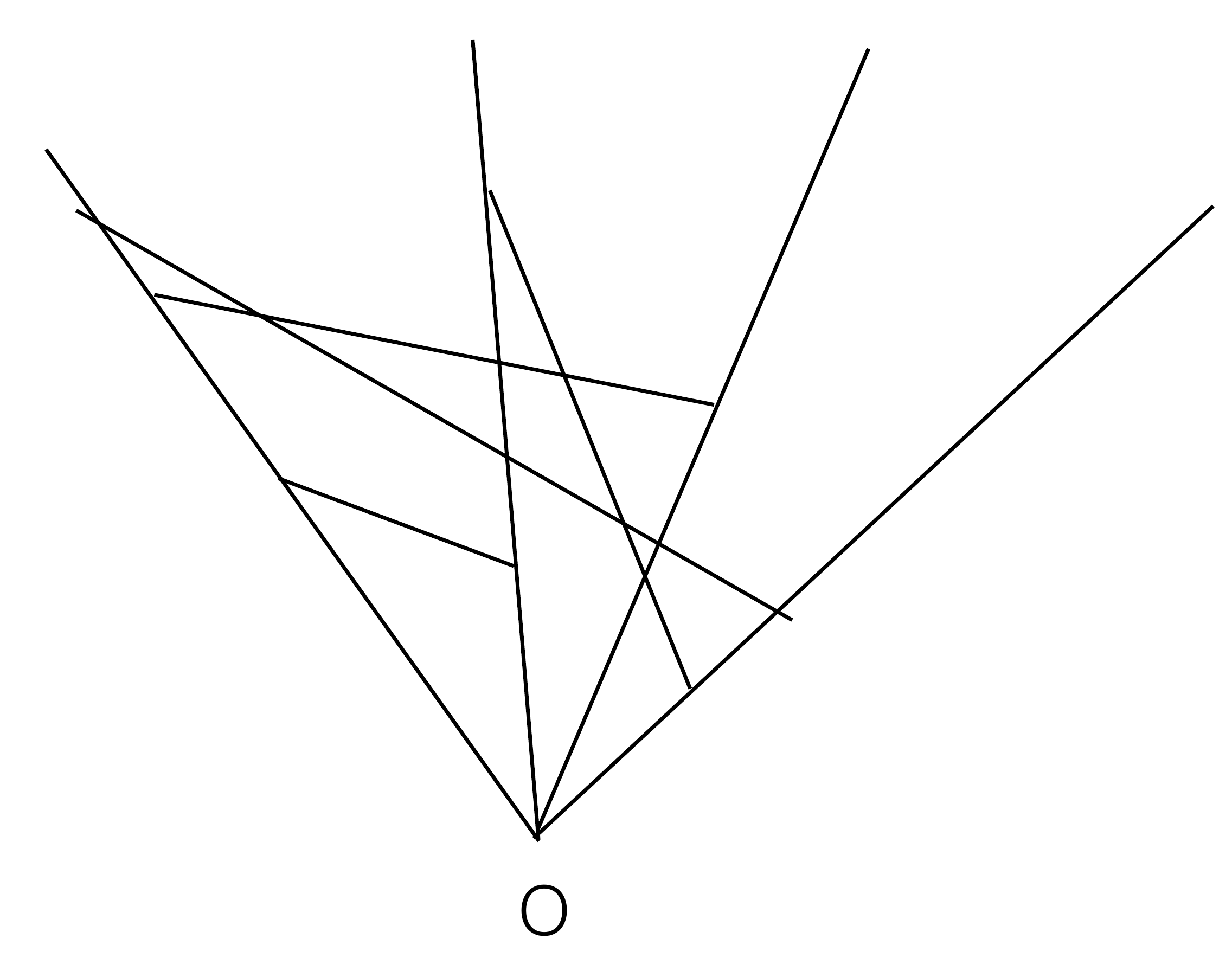}
    \end{minipage}
    \label{fig:general.example} 
  }  \caption{An illustration of different search domains.}
\end{figure}

\subsection{Related work}
\label{subsec:related.work}

Expanding search on a network was introduced in~\cite{AL:expanding}, with the focus on the Bayesian problem of minimizing the
expected search time against a known Hider distribution. In a followup paper~\cite{AL:19.general}, the same authors
studied expanding search on general networks and gave two strategy classes that have expected search times that are within
a factor close to 1 of the value of the game. Both these works apply to the unnormalized search time. 
For normalized objectives~\cite{ANGELOPOULOS201951} recently studied expanding search in a fixed (finite) graph in which the Hider can only hide on
vertices. In terms of finding a strategy of optimal deterministic competitive ratio~\cite{ANGELOPOULOS201951} 
showed that the problem is NP-hard, and gave a $4\ln 4$ approximation. Concerning the randomized competitive ratio, the same work
presented a strategy that is a $5/4$-approximation in the special case of tree graphs. 

The competitive ratio of pathwise search was first studied by Beck and Newman in the context of the linear search problem~\cite{beck:yet.more}
and later by Gal~\cite{gal:general,gal:minimax} for star search. For fixed graphs, assuming that the Hider can only hide on vertices, it is 
NP-hard to approximate the deterministic competitive ratio~\cite{koutsoupias:fixed}. The same paper also gave constant-factor approximations
for both the deterministic and the randomized competitive ratio, assuming the graph is undirected. Extensions to edge-weighted graphs
were studied in~\cite{DBLP:conf/ciac/AusielloLM00}, which also showed connections between graph searching and classic optimization
problems such as the Traveling Salesman problem and the Minimum Latency problem. The setting in which the search graph is not known
to the Searcher, but is rather revealed as the search progresses was studied in~\cite{fleischer:online}.

The exact and approximate competitive ratio of pathwise search has been studied in many settings, mostly assuming 
a star-like search domain. Examples include multi-Searcher strategies~\cite{alex:robots,Angelopoulos2016}, 
searching with turn cost~\cite{demaine:turn,Angelopoulos2017}, searching with probabilistic information~\cite{jaillet:online},
searching with upper/lower bounds on the distance of the Hider from the root~\cite{DBLP:journals/dam/HipkeIKL99,ultimate,revisiting:esa}, 
and searching for multiple hiders~\cite{multi-target,oil,hyperbolic}. All these works assume
that the search domain is either the unbounded line or the unbounded star.

\section{Preliminaries}
\label{sec:preliminaries}

%In this section we present some preliminary definitions and argue about the existence of 
%the deterministic and randomized competitive ratios in the problems we study. 
We consider a search domain that is represented by a connected network $Q$ which consists of vertices and arcs, and which has 
a certain vertex $O$ designated as its root. Moreover, $Q$ is endowed with Lebesgue measure corresponding to length. 
The measure of a subset $A$ of $Q$ is denoted by $\lambda(A)$, and in the case that $Q$ has finite measure, we will denote by $\mu =\lambda (Q)$ the total measure of $Q$.
This defines a metric on $Q$, where $d(x,y)$ is the length of the shortest path from $x$ to $y$. We write $d(x)$ for the distance $d(O,x)$ from $O$ to $x$. We denote by $\textrm{deg}_Q(v)$ the degree of $v$ in $Q$, namely the number of arcs incident to $v$.

We do not limit ourselves to bounded networks, but make the standing assumption that the network $Q$ satisfies the condition that there exists some integer $M$ such that for any $r > 0$,
\begin{align}
|\{x \in Q: d(x) = r \}| \le M. \label{eq:cond}
\end{align}
That is, there are at most $M$ points at distance $r$ from $O$. Any network with a finite number of arcs automatically satisfies this condition.  We will see that this condition ensures that the competitive ratio exists. As an example of an unbounded network, if $Q$ is an $m$-ray star, we have 
$|\{x \in Q: d(x) = r \}|=m$, for all $r$, hence this quantity is bounded. In contrast, if $Q$ is an unbounded full binary
tree in which there are $2^i$ vertices at distance $i$ from the root, for all $i \in \mathbb{N}^+$, then this quantity is unbounded, and this implies the competitive ratio is also unbounded. Indeed, any deterministic search strategy (pathwise or expanding) on this network must take at least time $2^i$ to reach all points at distance $i$ from the root, so the deterministic competitive ratio would be at least $2^i/i$, which is unbounded. We will see later (Proposition~\ref{prop:cont-random}) that this implies that the randomized competitive ratio of this tree network is also unbounded.

Given a network $Q$, and any $r\geq 0$, we denote the closed disc of radius $r$ around $O$ by $Q[r]=\{ x\in Q:d(x)\leq r\}$. Let $r_{\max}=\max_{x \in Q} d(x)$ be the distance of the furthest point in $Q$ from $O$, where $r_{\max}= \infty$ if $Q$ is unbounded. We define the 
real function $f_Q:[0,r_{\max}] \rightarrow \mathbb{R}$ given by $f_Q(r)=\lambda (Q[r])$, so $f_Q(r)$ is the measure of the set of points 
at distance no more than $r$ from the root.

We begin with preliminary definitions and results concerning expanding search, since it is a more recent paradigm, and somewhat 
more subtle to define. We then explain how these definitions change in what concerns pathwise search.

\subsection{Expanding search}
\label{subsec:prelim.expanding}

In expanding search, we allow the search to move at no cost over any part of the network that it has previously explored. This is formalized in the following definition.

\begin{definition}[\cite{AL:expanding}]
	\label{def:cont-exp}
	An expanding search on a network $Q$ with root $O$ is a family of
	connected subsets $S(t)\subset Q$ (for $0 \leq t\leq \mu $) satisfying:
	(i) $S(0)=O$; (ii) $S(t)\subset S(t^{\prime })$ for all $t\le t^{\prime }$; and (iii)
	$\lambda (S(t))=t$ for all $t$.
\end{definition}

If the context is clear, we will refer to an expanding search as a \em{search} \em{search strategy. For a given expanding search $S$ of $Q$ and a point $H\in Q$, let $T(S,H)=\min \{t:H\in S(t)\} $ be the \emph{(expanding) search time} of $H$ under $S$.
This was shown to be well defined in~\cite{AL:expanding}.
For $H \neq O$, let $\hat{T}(S,H)$ be the ratio $T(S,H)/d(H)$ of the search time of $H$ to the distance of $H$ from the root. 
We refer to $\hat{T}(S,H)$ as the \emph{normalized search time}.
It is convenient to define $\hat{T}(S,O)$ to be equal to $0$.

\begin{definition}	\label{def:det.search.cont}
	The deterministic competitive ratio $\sigma_S=\sigma_S(Q)$ of a deterministic expanding search $S$ of a network $Q$ is given by
	\[
  \sigma_S(Q) = \sup_{H \in Q} \hat{T}(S,H).
  \]
	The (deterministic, expanding) competitive ratio, $\sigma = \sigma(Q)$ of $Q$ is given by
	\[
	\sigma(Q) = \inf_S \sigma_S(Q),
	\]
	where the infinum is taken over all search strategies $S$. If $\sigma_S=\sigma$ we say that $S$ is optimal.
\end{definition}

Note that the competitive ratio of a strategy $S$ may be infinite. For example, suppose that $Q$ consists of two unit-length arcs $a$ and $b$ meeting at the root and suppose $S$ searches $a$ first and then $b$. If $H$ lies on the arc $b$ at distance $x$ from the root then $\hat{T}(S,H)=(1+x)/x=1+1/x \rightarrow \infty$ as $x \rightarrow 0$. It is not immediately obvious whether or not the competitive ratio of a  network is finite in general, 
but we will show in Section~\ref{sec:deterministic} that this is indeed the case, by explicitly giving the optimal search strategy for any network.

In addition, we consider {\em randomized} search strategies: that is, search strategies that are chosen according to some probability distribution. 
We denote randomized strategies by lower case letters, and for randomized strategies $s$ and $h$ for the Searcher and the Hider, respectively, 
we denote the \emph{expected search time} 
by $T(s,h)$ and the~\emph{expected normalized search time} by $\hat{T}(s,h)$.

\begin{definition} 	\label{def:rand.search.cont}
	The randomized competitive ratio $\rho_s=\rho_s(Q)$ of a randomized expanding search $s$ of a network $Q$ is given by
	\[
	\rho_s(Q) = \sup_{H \in Q} \hat{T}(s,H).
	\]
	The randomized competitive ratio, $\rho = \rho(Q)$ of $Q$ is given by
	\[
	\rho(Q) = \inf_s \rho_s(Q),
	\]
	where the infimum is taken over all possible randomized search strategies $s$. If $\rho_s=\rho$ we say
	that $s$ is optimal.
\end{definition}
When clear from context, we omit $Q$ for simplicity, e.g., we will use $\sigma_S$ instead of $\sigma_S(Q)$.

We can view the randomized competitive ratio of a network as the {\em value} of the following zero-sum game $\Gamma(Q,O)$.
A strategy $S$ for the Searcher is a search strategy as described above and a strategy $H$ for the Hider is a point on $Q$. The payoff of the game is the normalized search time $\hat{T}(S,H)$. For randomized strategies $s$ and $h$ of the Searcher and Hider, respectively, the expected payoff is denoted by $\hat{T}(s,h)$.

In~\cite{AL:expanding} the authors considered a similar zero-sum game on finite networks in which the players' strategy sets are the same but the payoff is the unnormalized search time $T(S,H)$. They showed that the strategy sets are compact with respect to the uniform Hausdorff metric and that $T(S,H)$ is lower semicontinuous in $S$ for fixed $H$. Since $d(H)$ is a constant for fixed $H$, it follows that $\hat{T}(S,H) = T(S,H)/d(H)$ is also lower semicontinuous in $S$ for fixed $H$, and by the Minimax Theorem~\cite{AG:minimax}, we have the following theorem.

\begin{theorem}
	Let $Q$ be a finite network with root $O$. The game $\Gamma(Q,O)$ has a value, which is equal to the randomized competitive ratio $\rho(Q)$. The Searcher has an optimal mixed strategy (with competitive ratio $\rho(Q)$) and the Hider has $\varepsilon$-optimal mixed strategies.
	\label{thm:existence}
\end{theorem}

It is not so straightforward to show that the game has a value if $Q$ is unbounded. Nonetheless, this is not important for our analysis, and we will rely on the following general result for zero-sum games that for any mixed Hider strategy $h$,
\begin{align}
\rho(Q) \ge \sup_{S} \hat{T}(S,h), \label{eq:lb}
\end{align}
where the supremum is taken over all search strategies $S$.

\subsection{Pathwise search}
\label{subsec:prelim.pathwise}

For pathwise search, which is the usual search paradigm, the Searcher follows a continuous, unit-speed path: that is a trajectory $S:[0,\infty)\rightarrow Q$ with $S(0)=O$ and $d(S(t_1),S(t_2)) \le t_2-t_1$ for all $t_1 < t_2$. For such a pathwise search $S$ and a point $H$ on $Q$, the (pathwise) search time $T(S,H)$ of $H$ under $S$ is the first time that $H$ is reached by the Searcher, i.e., $\min \{t \ge 0: S(t)=H \}$. The concepts of deterministic and randomized search times, as well as the deterministic and randomized competitive ratios are defined analogously to Definitions~\ref{def:det.search.cont} and~\ref{def:rand.search.cont}. 

As in the case of expanding search, we may view the randomized competitive ratio of a network as the value of a game played between a minimizing Searcher and a maximizing Hider where the payoff is the search time. In the case of finite networks, it is easy to show that the value exists, 
whereas for unbounded networks, it is again the inequality~(\ref{eq:lb}) which will be most essential in our analysis.

%\notes{I think we need more about the existence of equilibria here, but I am not sure about the references. Also, we need to decide, in terms of presentation, whether we assume networks that can be unbounded, otherwise some concepts will not work, such as $r_{\max}$ in the next section.}

\section{Deterministic, expanding competitive ratio}
\label{sec:deterministic}

In this section we show how to obtain an expanding search of optimal deterministic ratio, using a ``water filling'' principle.
Informally, the network is searched in such a way that the set of points that have been searched at any given time form an expanding disc around $O$.
Recall the definition of $f_Q$ from Section~\ref{sec:preliminaries}. $f_Q$ is piecewise linear and strictly increasing so has an inverse $g_Q$. The interpretation is that $g_Q(t)$ is the unique radius $r$ for which $Q[r]$ has measure $t$.

\begin{definition}
	For a network $Q$ with root $O$, consider the expanding search $S^*$ defined by $S^*(t)=Q[g_Q(t)]$ for $0 \le t \le r_{\max}$.
\end{definition}
Thus, $S^*(t)$ is an expanding disc of radius $g_Q(t)$. It is easy to verify that $S^*$ is indeed
an expanding search. First, we note that $S^*(t)$ is connected, since $Q[r]$ is always connected.
It also trivially satisfies (i) and (ii) from Definition \ref{def:cont-exp}, and (iii) is also satisfied since
\[
\lambda (S^*(t)) = \lambda (Q[g_Q(t)]) = f_Q(g_Q(t)) =t.
\]
We will show that $S^*$ attains the optimal competitive ratio. First, note that the search time of a point $H\in Q$ under $S^*$ is the unique time $t$ such that $S^*(t)=Q[d(H)]$, so $T(S^*,H)=\lambda (Q[d(H)]) =f_Q(d(H)) $. Hence, the competitive ratio of $S^*$ is
\begin{align}
\sigma _{S^*} &=\sup_{H\in Q-\{O\}}\frac{f_Q(d(H))}{d(H)} \nonumber \\
&=\sup_{r > 0}\frac{f_Q(r)}{r} \label{eq:sigmaS*}.
%&=\frac{1}{\inf_{t >0}\frac{g (t)}{t}}.
\end{align}
%\label{eq:continuous}
This has an intuitive interpretation as follows: if we draw the graph of $f_Q(r)$ then the competitive ratio is the slope of the steepest straight line through the root that intersects with the graph of $f_Q(r)$. Condition~(\ref{eq:cond}) ensures that $\sigma_{S^*}$ is finite for unbounded networks, since $f_Q(r) \le Mr$ for all $r$.

\begin{theorem} \footnote{This theorem appeared without proof as Theorem~6 of \cite{stacs-expanding}.}
	\label{thm:cont-determ}
    The expanding search $S^*$ is optimal and the competitive ratio $\sigma$ of a network $Q$ with root $O$ is given by
	\begin{equation}
	\sigma = \sup_{r > 0} \frac{f_Q(r)}{r}.
	\label{eq:cont-sigma}
	\end{equation}
\end{theorem}

\begin{proof}
	let $S$ be an optimal search, and let $t(r) = \min \{t>0: Q[r] \subset S(t) \}$ be the first time that $S$ contains $Q[r]$. Then the maximum search time of any point $H$ at some fixed distance $r$ from $O$ is $t(r)$, and it follows that $\sigma=\sigma_S$ is given by
	\[
	\sigma_S = \sup_{r>0} \frac{t(r)}{r}.
	\]
Clearly, $ t(r) \ge f_Q(r)$, so $\sigma_{S^*} \le \sigma_S$, by~(\ref{eq:sigmaS*}). The optimality of $S^*$ and the expression for $\sigma$ follows.
\end{proof}

\section{Randomized, expanding competitive ratio}
\label{sec:random}

In this section we study the randomized competitive ratio of expanding search, which is significantly more 
challenging to analyze than the deterministic one. 
We begin by showing that the randomized competitive ratio is at most
half the deterministic competitive ratio and that there exist networks for which this bound is tight (Section~\ref{subsec:randomized.half}). 
In Section~\ref{subsec:randomized.better}
we give a Hider strategy that allows us to get useful lower bounds on the randomized competitive ratio. We also obtain
bounds on $\rho$ that are parameterized by the function $f_Q$, from which 
we can deduce the randomized competitive ratio for networks with concave $f_Q$. In Section~\ref{subsec:randomized.y}
we show that the randomized strategy may have a quite complex statement, even for very simple networks that
consist only of three arcs. We address this difficulty in Section~\ref{subsec:randomized.5/4}, in which we 
give a strategy that is within a factor at most $5/4$ of the optimal randomized competitive ratio, for all networks.

\subsection{A simple approximation of the randomized competitive ratio}
\label{subsec:randomized.half}

Recall that $S^*$ denotes the optimal deterministic search strategy of Section~\ref{sec:deterministic}.
\begin{proposition} \footnote{This proposition appeared without proof as Proposition 7 of~\cite{stacs-expanding}.}
	\label{prop:cont-random}
	For a network $Q$ with root $O$, the randomized competitive ratio $\rho$ satisfies
	\[
	\sigma/2 \le \rho \le \sigma.
	\]
	Furthermore, the bounds are tight, in the sense that they are the best possible.
\end{proposition}

\begin{proof}
The right-hand inequality is clear, since every deterministic search strategy
is also a randomized search strategy.
To prove the left-hand inequality, we first observe that since $S^*$ is an optimal deterministic search, for any $\varepsilon>0$, we can find some point $H$ on $Q$ such that $\hat{T}(S^*,H) \ge \sigma - \varepsilon$. Let $r = d(H)$ so that $\sigma \le f_Q(r)/r + \varepsilon $. Let $h$ be the Hider strategy that hides on $Q[r]$ uniformly: that is, it chooses a subset of $Q[r]$ with probability proportional to the measure of that subset. For any search strategy $S$, the expected search time $T(S,h)$ is at least $\lambda(Q[r])/2$, so
\begin{align*}
\rho &\ge \sup_S \hat{T}(S,h) \\
&\ge \frac{\lambda(Q[r])/2}{r} \mbox{ (since every point in $Q[r]$ is at distance no more than $r$ from the root)} \\
&=  \frac{f_Q(r)}{2r} \\
&\ge \frac{\sigma - \varepsilon}{2}.
\end{align*}
Since $\varepsilon$ can be arbitrarily small, it follows that $\rho \ge \sigma/2$. 

We will now argue that both bounds are tight.
This is trivially true for the right-hand inequality since the network consisting of one arc with the root at its end has the same deterministic and randomized competitive ratio.

For the left-hand inequality, consider the network depicted in Figure \ref{fig:cont-det}. The normalized search time $\hat{T}(S^*,H)$ is maximized at leaf nodes $X$, so that $\sigma = \hat{T}(S^*,X) = (n+n^2)/(n+1) = n$.

Consider now the randomized strategy $s$ that searches the arc of length $n$ first before searching the remainder of the arcs in a uniformly random order. Then all points $H$ at distance no greater than $n$ have expected normalized search time $1$; a point $H$ at distance $d > n$ has
\[
\hat{T}(s,H) = \frac{d + (n^2-1)/2}{d} \le 1+ \frac{(n^2-1)/2}{n} \le 1 + n/2,
\]
so $\rho \le 1 + n/2 =1+\sigma/2$. Since $\rho \ge \sigma/2$, we must have that $\sigma /\rho \rightarrow 2$, as $n \rightarrow \infty$.
\begin{figure}[htb!]
	\centerline{\includegraphics[scale=0.4]{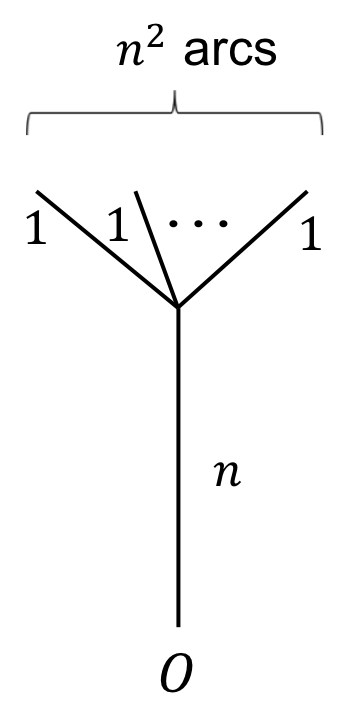}}
	\caption{A network for which $\rho \approx \sigma/2$.}
	\label{fig:cont-det}
\end{figure}
\end{proof}

A corollary of Proposition~\ref{prop:cont-random} is  that the ``water-filling'' search $S^*$ approximates the optimal randomized search by a factor of $2$.

\subsection{A Hider strategy, and lower bounds on the randomized competitive ratio}
\label{subsec:randomized.better}

For a general network $Q$, let $A$ be a connected subset. Let $d(A)$ be the distance from $O$ to $A$ and let $u_A$ be the Hider strategy (probability measure) that hides uniformly on $A$, so that $u_A(X) = \lambda(X)/\lambda(A)$ for a measurable subset $X \subset A$. Then denote the average distance from $O$ to points in $A$ by $\overline{d}(A)=\int_{x \in A}d(x) du_A(x)$.
\begin{theorem}\label{thm:lbound}
		Consider the Hider strategy $h_A$ given by 
		\[
		dh_A(x) = \frac{d(x)}{\overline{d}(A)} du_A(x).
		\]
		By adopting the strategy $h_A$, the Hider ensures that the randomized competitive ratio $\rho$ satisfies
	\[
	\rho \ge \frac{d(A)+\lambda(A)/2}{\overline{d}(A)}.
	\]
\end{theorem}

\begin{proof}
Let $S$ be any search strategy, and note that $T(S,u_A)\ge d(A) + \lambda(A)/2$. We have
	\begin{align*}
	\rho \ge \hat{T}(S,h_A) & = \int_{x \in A} \frac{T(S,x)}{d(x)} dh_A(x) \\
	& = \frac{1}{\overline{d}(A)} \int_{x \in A} T(S,x) du_A(x) \\
	& = \frac{1}{\overline{d}(A)} T(S,u_A)\\
	&\ge \frac{d(A)+\lambda(A)/2}{\overline{d}(A)}.
	\end{align*}
\end{proof}
To illustrate the applicability of Theorem~\ref{thm:lbound}, we show how to obtain, in a different way, the corollary of Proposition~\ref{prop:cont-random} that the optimal deterministic search strategy approximates the optimal randomized strategy by a factor of $2$.
\begin{corollary}
	The optimal deterministic search $S^*$ approximates the randomized competitive ratio by a factor of~$2$. 
\end{corollary}
\begin{proof}
	Let $\varepsilon >0$ and let $x$ be a point of $Q$ such that $\sigma \le \hat{T}(S^*,x)+\varepsilon/2 =\lambda(A)/d(x) + \varepsilon/2$, where $A \equiv Q[d(x)]$ is the set of all points at distance at most $d(x)$. By Theorem~\ref{thm:lbound}, $\rho \ge \lambda(A)/(2\overline{d}(A))$, so
	\[
	\frac{\sigma}{\rho} \le \frac{\lambda(A)/d(x) + \varepsilon/2}{\lambda(A)/(2\overline{d}(A))} = \frac{2\overline{d}(A)}{d(x)} + \frac{ \varepsilon \overline{d}(A)}{\lambda(A)} \le 2 + \varepsilon.
	\]
Since $\varepsilon$ can be arbitrarily small, the corollary follows.
\end{proof}

More importantly, Theorem~\ref{thm:lbound} allows us to obtain the following lower bound on the randomized competitive ratio of $Q$.
\begin{lemma}
For any network $Q$ with root $O$, it holds that $\rho\geq \deg_Q(O)$.
\label{lemma:randomized.degree}
\end{lemma}
\begin{proof}
Let $E_O$ denote the set of arcs in $Q$ that are incident with $O$.
Fix $r_0>0$ such that $r_0\leq \min_{e \in E_O} \lambda(e)$; clearly, such an $r_0$ must exist. Let $A=Q[r_0]$ be the ball of points in
$Q$ that are at distance at most $r_0$ from $O$, and let $h_A$ be the Hider strategy associated with $A$, and defined as in the 
statement of Theorem~\ref{thm:lbound}. We calculate the average distance $\overline{d}(A)$ from $O$ to points in $A$ by writing
  \begin{align*}
  \overline{d}(A) &= \int_{0}^{r_{0}} 1- u_A(Q[r]) dr
  = \int_{0}^{r_{0}} 1- \frac{r}{r_{0}}dr 
  = \frac{r_{0}}{2}. 
  \end{align*}
Moreover, from the definition of $A$, we have that $\lambda(A)=\deg_Q(O) \cdot r_0$. By Theorem~\ref{thm:lbound}, we have 
  \[
  \rho \ge \frac{\lambda(A)}{2\overline{d}(A)} = \deg_Q(O). 
  \]
\end{proof}
The above lemma implies a tight bound on the randomized competitive ratio for all networks $Q$ for which the function $f_Q$
is concave, as shown in the following corollary.
\begin{corollary}
For any network $Q$ for which $f_Q$ is concave, we have that $\sigma=\rho=\deg_Q(O)$, and strategy $S^*$ is an optimal randomized strategy. 
\label{cor:concave}
\end{corollary}
\begin{proof}
The lower bound on $\rho$ follows from Lemma~\ref{lemma:randomized.degree}. For the upper bound, 
by~(\ref{eq:cont-sigma}), we have $\rho \le \sigma = \sup_{r>0} f_Q(r)/r$, and for any network for which $f_Q$ is concave, it holds
that $\sup_{r>0} f_Q(r)/r=\deg_Q(O)$. 
\end{proof}
An example of a network for which $f_Q$ is concave is depicted in Figure~\ref{fig:concave}, along with a plot of its
function $f_Q$.
\begin{figure}[htb!]
  \begin{center}
  \includegraphics[scale=0.25]{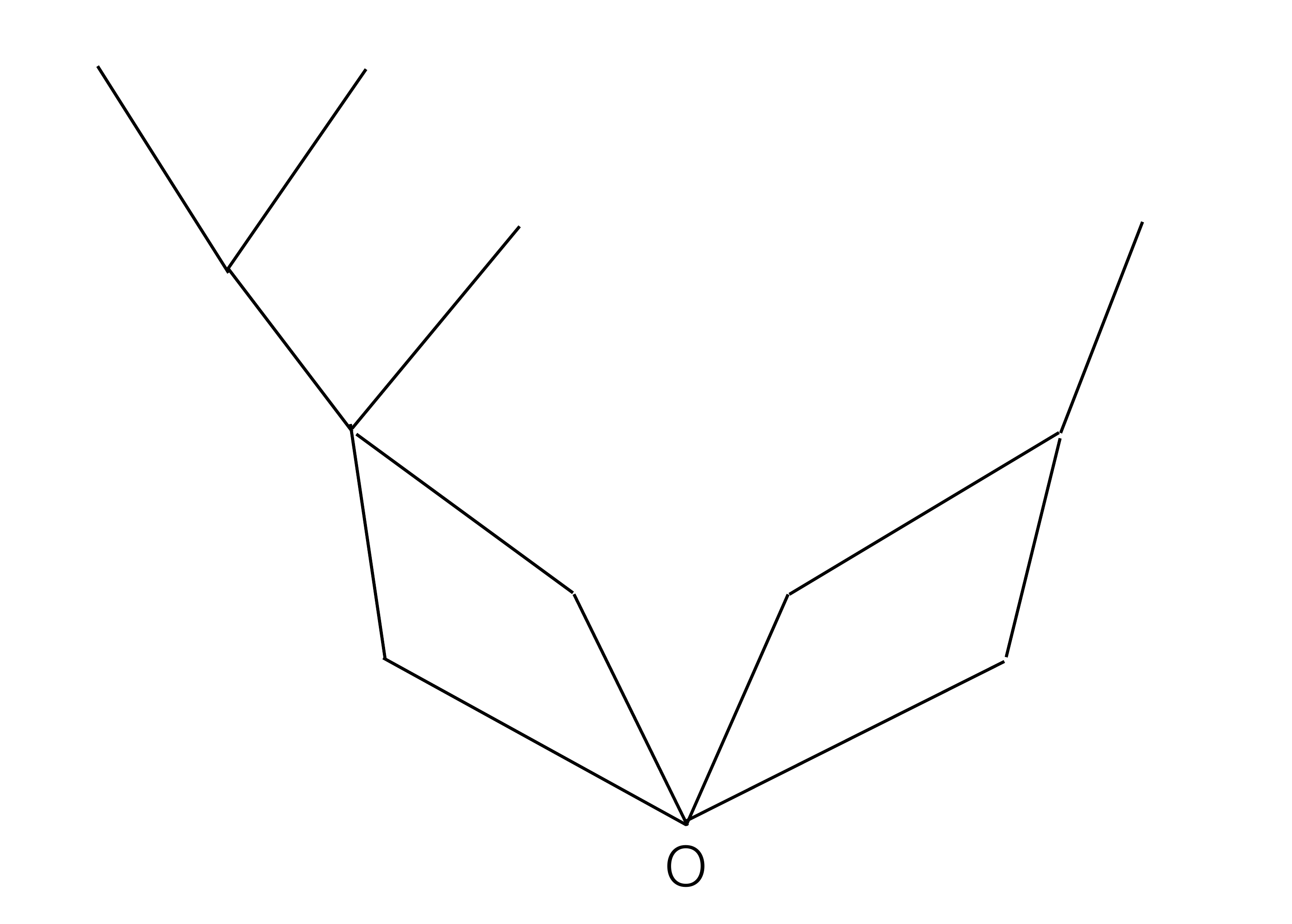} \qquad 
  \includegraphics[scale=0.2] {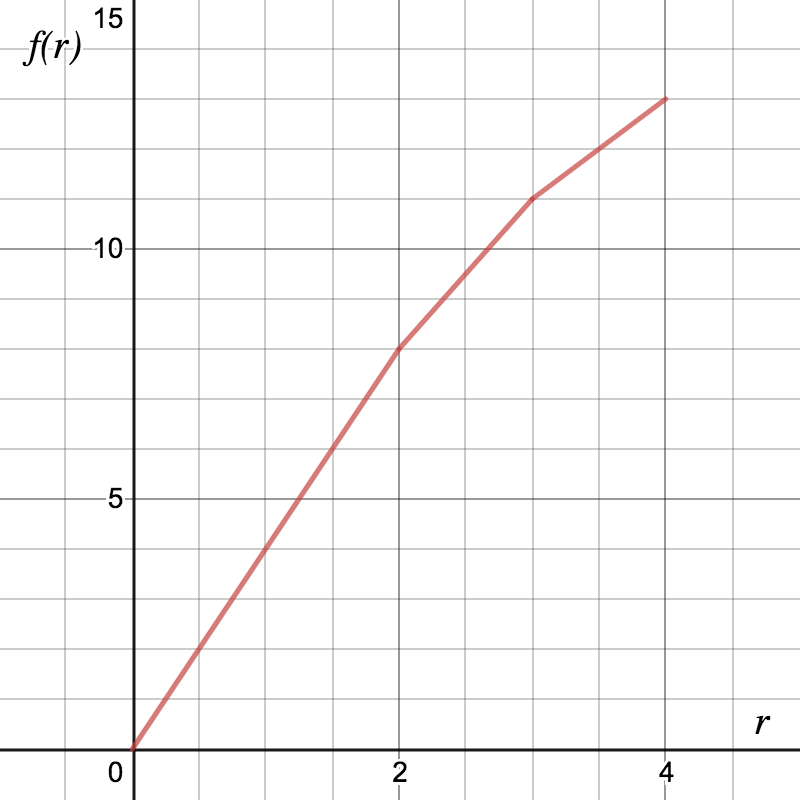}  
  \caption{An example of a network $Q$ (left) and the function $r\mapsto f(r) \equiv f_Q(r)$ (right). All arcs in $Q$ are unit-length.}
  \label{fig:concave}
    \end{center}
\end{figure}

More generally, we have established the following approximation.
\begin{corollary}
Suppose that for the network $Q$ it holds that $\sup_{r>0}f_Q(r)/r \leq \alpha \deg_{Q}(O)$, for some $\alpha>1$. Then $S^*$
approximates the optimal randomized ratio of $Q$ within a factor of at most $\alpha$. 
\label{cor:randomized.approximation.f}
\end{corollary}

\subsection{Optimal randomized strategies are complex: $Y$-networks}
\label{subsec:randomized.y}

We now consider a class of the simplest networks for which the function $f_Q$ is not concave, and thus Corollary~\ref{cor:concave} does not
apply. In particular, we consider the {\em $Y$-network} depicted in Figure~\ref{fig:Ynet} consisting of a node $v$ which is incident 
to three arcs of lengths $1$, $L$ and $M \ge L$. The root node is the other endpoint of the arc of length $1$. We refer to the arc of 
length $L$ as the ``left arc'' and the arc of length $M$ as the ``right arc''.

\begin{figure}[htb!]
  \begin{center}
    \includegraphics[scale=0.4]{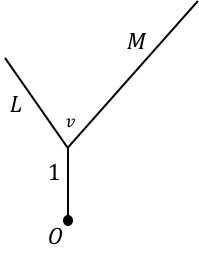}
    \caption{The $Y$-network.}
    \label{fig:Ynet}
  \end{center}
\end{figure}

Clearly the optimal Hider strategy on the $Y$-network will hide on the arc incident to the root with probability $0$. Let $A$ be the subset of $Q$ consisting of all the points on the left arc at distance at most $x$ from $v$ and all the points on the right arc at distance at most $y$ from $v$. From
Theorem~\ref{thm:lbound}, by using the strategy $h_A$, the Hider ensures that the competitive ratio is at least
\begin{align*}
\rho &\ge \frac{d(A)+\lambda(A)/2}{\overline{d}(A)}\\
&\frac{1+ (x+y)/2}{1+(x/(x+y))(x/2)+(y/(x+y))y/2}\\
& = 1+\frac{2xy}{x(x+2)+y(y+2)}.
\end{align*}  
By elementary calculus, this bound is maximized for $x=L$ and $y=\max\{M, \sqrt{L(L+2)}\}$, giving
\begin{align}
\rho &\ge V:=
1+ \frac{2LM'}{L(L+2)+M'(M'+2)},
\label{eq:Y-lb}
\end{align}
where $M'=\max\{M, \sqrt{L(L+2)}\}$.

We show that the expression $V$ given in~(\ref{eq:Y-lb}) is indeed the randomized competitive ratio by giving an optimal Searcher strategy. The optimal Searcher strategy mixes between four different strategies which we list below. (Each strategy begins by searching the arc incident to the root, so we do not mention this part of the search.)
\begin{enumerate}
  \item[$A$.] Search the left arc first then search the right arc.
  \item[$B$.] Search the right arc up to length $M'$ first then the left arc then search the remainder of the right arc.
  \item[$C$.] Search the left arc and the right arc at the same time, at speeds proportional to $L$ and $M$ respectively, until the whole of the left arc has been searched, then search the remainder of the right arc. In other words, in the time interval $[1,1+t]$, search $tL/(L+M')$ of the left arc and $tM'/(L+M')$ of the right arc, for $t \le L+M'$, then search the remainder of the right arc.
  \item[$D$.] Begin by searching the right arc, but at some time chosen uniformly at random between $0$ and $M'$, search the whole of the left arc before completing the search of the right arc.
\end{enumerate}

In Table~\ref{tab:Y-search} we list probabilities that the Searcher should choose each of these four strategies along with the expected search time of a point at distance $a \le L$ from $v$ on the left arc and a point at distance $b \le M'$ from $v$ on the right arc.

\begin{table}[h]
  \begin{center}
    \begin{tabular}{|l|c|c|c|}
      \hline
      {\bf Search strategy}& {\bf Probability} & {\bf Expected search time on left} & {\bf Expected search time on right} \\
      \hline
      $A$ & $\frac{2M}{L(L+2)+M'(M'+2)}$ & $1+ a$ & $1+L+b$  \\ 
      $B$ & $\frac{L^2+2L-M^2}{L(L+2)+M'(M'+2)}$ &  $1+M'+a$ & $1+b$ \\ 
      $C$ & $\frac{2L^2}{L(L+2)+M'(M'+2)}$ &  $1 + \frac{a}{L}(L+M')$ & $1 + \frac{b}{M'}(L+M')$ \\
      $D$ & $\frac{2M^2-2L^2}{L(L+2)+M'(M'+2)}$ & $1+\frac{M'}{2}+a$ & $1 + \frac{b}{M'}(L+M')$\\
      \hline
    \end{tabular}
    \caption{An optimal search strategy for the $Y$-network.}
    \label{tab:Y-search}
  \end{center}
\end{table} 
 A simple calculation shows that for points on the left arc at distance $a \le L$ from $v$, the expected search time is $V(1+a)$ and for points on the right arc at distance $b \le M'$ from $v$, the expected search time is $V(1+b)$. For points on the right arc at distance $b > M'$ from $v$ (if such points exist), the expected search time is $1+L+b$, and it is easy to show that this is strictly less than $V(1+b)$. Hence the randomized competitive ratio is $V$.

\subsection{A $5/4$-approximation of the randomized competitive ratio}
\label{subsec:randomized.5/4}

In this section we give a search strategy that is a $5/4$-approximation of the optimal randomized search. This is inspired by the strategy of~\cite{ANGELOPOULOS201951} for the discrete case, namely for searching in a given graph when the Hider can only hide at a vertex.

We first define the concept of a {\em Randomized Depth-First Search} ({\em RDFS}) of a tree $\mathcal{T}$. Let $S$ be any depth-first search of $\mathcal{T}$ and let $S^{-1}$ be the depth-first search that visits the leaf nodes of $\mathcal{T}$ in the reverse order from $S$. Then the randomized search $s$ that chooses between $S$ and $S^{-1}$ equiprobably is a RDFS of $\mathcal{T}$.

\begin{lemma} \label{lem:RDFS}
	Let $s$ be a RDFS of a tree $\mathcal{T}$. Then the expected time $T(s,H)$ at which a point $H \in \mathcal{T}$ is found by $s$ satisfies
	\[
	T(s,H) \le \frac{\lambda(\mathcal{T}) + d(H)}{2}.
	\]
\end{lemma}

\begin{proof}
	Suppose $s$ is an equiprobable mixture of the depth-first search $S$ and its reverse $S^{-1}$. Let $t_1 = T(S,H)$ and $t_2 = T(S^{-1},H)$. Then 
	\[
	T(s,H) = \frac{t_1+t_2}{2} = \frac{\lambda(S(t_1))+\lambda(S(t_2))}{2}.
	\]
	It is easy to see that $S(t_1) \cap S(t_2)$ is the path from $O$ to $H$, so 
	\[
	\lambda(S(t_1)) + \lambda(S(t_2)) = \lambda(S(t_1) \cup S(t_2)) + \lambda(S(t_1) \cap S(t_s)) \le \lambda(\mathcal{T}) + d(H).
	\]
	The lemma follows.
\end{proof}
Now we can define the randomized search that is a $5/4$-approximation. For an arbitrary network $Q$, let $Q_T$ be its shortest path tree. We will define a search of $Q_T$ which naturally translates to a search of $Q$. First we partition $Q_T$ into infinitely many randomly chosen subsets $R_{j}$, $j \in \mathbb{Z}$. To define the sets $R_j$, we choose numbers $d_j$ uniformly at random from the interval $[2^{j-1},2^j]$, and set $R_j: = \{x \in Q_T:  d_j \le d(x) <d_{j+1}\}$. We call the $R_j$ the {\em levels} of the search.

The {\em randomized doubling strategy} $s$ is defined as follows. At the start of the $j$th iteration, $\cup_{i < j}R_i$ has already been searched, and we shrink it to the root, so that now $R_j$ is a subtree of the resulting network. The $j$th iteration is then a RDFS of $R_j$. Note that this means that $s$ begins with infinitely small RDFS's, similarly to optimal strategies for the linear search problem, as studied in~\cite{gal:minimax}.

Before proving that the randomized doubling strategy is a $5/4$-approximation for the optimal randomized search, we first establish two technical lemmas. Let $Q_j = \{x \in Q_T: 2^{j-1} \le d(x) < 2^j \}$, for $j \in \mathbb{Z}$, and let $Q^j = \cup_{i \le j} Q_i$.

\begin{lemma} \label{lem:rho-bound}
	For any $j \in \mathbb{Z}$,
\[
\left(1-\frac{\overline{d}(Q^j)}{2^j}\right) \lambda(Q^j) \le 2^{j-1} \rho.
\]
\end{lemma}
\begin{proof}
	Applying Theorem~\ref{thm:lbound} to $Q_j$,
	\begin{align*}
	\left(1-\frac{\overline{d}(Q^j)}{2^j}\right)\lambda(Q^j) &\le \left(1-\frac{\overline{d}(Q^j)}{2^j}\right) \cdot 2 \overline{d}(Q^j) \rho.
	\end{align*}
	Regarding the right-hand side of the expression above as a quadratic in $\overline{d}(Q^j)$, it is maximized when $\overline{d}(Q^j)=2^{j-1}$, and the lemma follows.
\end{proof}

\begin{lemma} \label{lem:exp-measure}
	The expected measure of $Q_j \cap R_{j-1}$ is $(2 - \overline{d}(Q_j)/2^{j-1})\lambda(Q_j)$. 
\end{lemma}
\begin{proof}
A point $x \in Q_j$ is contained in $R_{j-1}$ if and only if $d_j > x$. This occurs with probability 
$(2^j - d(x))/2^{j-1}$. Therefore, the expected measure of $Q_j \cap R_{j-1}$ is
\[
\int_{x \in Q_j} \left( \frac{2^j - d(x)}{2^{j-1}} \right) \lambda(Q_j)~du(x) = \left(2- \frac{\overline{d}(Q_j)}{2^{j-1}} \right) \lambda(Q_j) .
\]
\end{proof}

\begin{theorem} \label{thm:5/4}
	The randomized doubling strategy $s$ is a $5/4$-approximation of the optimal randomized search. In particular, $\rho_s \le (5/4) \rho + 1/2$.
\end{theorem}

\begin{proof}
Suppose that the randomized competitive ratio of $s$ is maximized at some point $x$ which is contained in $Q_k$, for some $k$. Let $J$ be a random variable that takes the value $k-1$ or $k$ depending on whether $x$ is contained in $R_{k-1}$ or $R_k$, respectively. Let $L^J = \lambda(\cup_{i <J} R_i) + \lambda(R_J)/2$ be the random variable equal to the sum of half the measure of $R_J$ and the measure of all levels preceding $R_J$. Then, by Lemma~\ref{lem:RDFS}, the expected search time of $x$ is at most $\mathbb{E}(L^J) + d(x)/2$. Hence
\[
\frac{\rho_s}{\rho} \le \frac {(\mathbb{E}(L^J) + d(x)/2)/d(x)}{\rho} = \frac{\mathbb{E}(L^J)/d(x)}{\rho} + \frac{1}{2 \rho}.
\]
We just have to show that $\mathbb{E}(L^J)/(d(x)\rho) \le 5/4$. Let $L^J_1, L^J_2$ and $L^J_3$ be the contributions to $L^J$ from $Q_{k-1},Q_{k}$ and $Q_{k+1}$, respectively, so that $L^J=\lambda(Q^{k-2}) + L^J_1+L^J_2+L^J_3$.

We first compute $\mathbb{E} (L^J_1)$. Note that if $d_k \le x$, which happens with probability $(d(x)-2^{k-1})/2^{k-1}$, then $J=k$ so that $R_J$ is disjoint from $Q_{k-1}$. In this case $L^J_1=\lambda(Q_{k-1})$. Otherwise, with probability $(2^k-d(x))/2^{k-1}$, we have that $J=k-1$, and $L^J_1$ is equal to the sum of half the expected measure of $Q_{k-1} \cap R_{k-1}$ and the measure of $Q_{k-1} \cap R_{k-2}$, or equivalently, the sum of $\lambda(Q_{k-1})/2$ and half the expected measure of $Q_{k-1} \cap R_{k-2}$. Applying Lemma~\ref{lem:exp-measure}, with $j=k-1$, this is equal to
\[
\frac{\lambda(Q_{k-1})}{2} + \frac{1}{2} \left( 2- \frac{\overline{d}(Q_{k-1})}{2^{k-2}} \right) \lambda(Q_{k-1}) = \left(\frac{3}{2} - \frac{\overline{d}(Q_{k-1})}{2^{k-1}}  \right)\lambda(Q_{k-1}) .
\] 
Putting this together,
\begin{align}
\mathbb{E}(L^J_1) &= \left( \frac{d(x)-2^{k-1}}{2^{k-1}} \right)\lambda(Q_{k-1}) + \left( \frac{2^k-d(x)}{2^{k-1}} \right) \left(\frac{3}{2} - \frac{\overline{d}(Q_{k-1})}{2^{k-1}}  \right)\lambda(Q_{k-1}) \nonumber \\
& = \left( \left(2 - \frac{d(x)}{2^k} \right) - \left(2 - \frac{d(x)}{2^{k-1}} \right) \frac{\overline{d}(Q_{k-1})}{2^{k-1}} \right) \lambda(Q_{k-1}). \label{eq:L1}
\end{align}
Next, we consider $L^J_3$. With probability $(2^k-d(x))/2^{k-1}$, we have that $d_k > d(x)$, so that $J=k-1$, and $R_J$ is disjoint from $Q_{k+1}$. In this case, $L^J_3$ is zero. Otherwise, $J=k$, and $L^J_3$ is equal to half the expected measure of $Q_{k+1} \cap R_k$. Applying Lemma~\ref{lem:exp-measure} again, this time with $j=k+1$, gives
\begin{align}
\mathbb{E}(L^J_3) =\left( \frac{d(x) - 2^{k-1}}{2^{k-1}} \right) \left( 1 - \frac{\overline{d}(Q_{k+1})}{2^{k+1}} \right)\lambda(Q_{k+1}). \label{eq:L3}
\end{align}
Lastly, we consider $L^J_2$. Denote by $Q_k[d]$ the set of points in $Q_k$ at distance at most $d$ from $O$. If $d_k \le d(x)$, then $J=k-1$ and $L^J_2 = \lambda(Q_k[d_k]) + \lambda(Q_k - Q_k[d_k])/2$. If $d_k > d(x)$ then $J=k$ and $L^J_2 = \lambda(Q_k[d_k])/2$. Integrating over all possible value of $y=d_k$, we obtain
\begin{align*}
\mathbb{E} (L^J_2) & = \frac{1}{2^{k-1}} \int_{2^{k-1}}^{d(x)}  \lambda(Q_k[y]) + \lambda(Q_k - Q_k[y])/2 ~dy + \frac{1}{2^{k-1}} \int_{d(x)}^{2^k} \lambda(Q_k[y])/2~dy  \\
&= \frac{1}{2^{k-1}} \int_{2^{k-1}}^{d(x)} \lambda(Q_k)/2~dy + \frac{1}{2}\int_{2^{k-1}}^{2^k}  \frac{1}{2^{k-1}} \lambda(Q_k[y])~dy.
\end{align*}
Now, the second integral above is equal to the expected measure of $Q_k \cap R_{k-1}$, and using Lemma~\ref{lem:exp-measure} with $j=k$ gives
\begin{align}
\mathbb{E}(L^J_2) &= \left( \frac{d(x) - 2^{k-1}}{2^k} \right) \lambda(Q_k)  + \frac{1}{2} \left( 2 - \frac{\overline{d}(Q_k)}{2^{k-1}} \right) \lambda(Q_k) \nonumber \\
& = \left( \left( \frac{d(x)}{2^k} + \frac{1}{2} \right)  - \frac{ \overline{d}(Q_k)}{2^k} \right)\lambda(Q_k). \label{eq:L2}
\end{align}
Substituting Equations~(\ref{eq:L1}), (\ref{eq:L3}) and~(\ref{eq:L2}) in $\mathbb{E}(L)=\lambda(Q^{k-2})+\mathbb{E}(L^J_1)+\mathbb{E}(L^J_2)+\mathbb{E}(L^J_3)$ and rearranging, we obtain
\begin{align*}
\mathbb E(L^J) &= \left(1- \frac{d(x)}{2^k} \right) \left( \frac{\overline{d}(Q^{k-2})}{2^{k-2}} - 1 \right)\lambda(Q^{k-2}) + \left(\frac{3}{2} - \frac{d(x)}{2^{k-1}} \right) \left( 1 - \frac{\overline{d}(Q^{k-1})}{2^{k-1}} \right) \lambda(Q^{k-1}) \\
&  + \left( \frac{3}{2} - \frac{d(x)}{2^k} \right) \left( 1 - \frac{\overline{d}(Q^k)}{2^k} \right) \lambda(Q^k) + \left( \frac{d(x)}{2^{k-1}} - 1 \right)\left( 1 - \frac{ \overline{d}(Q^{k+1})}{2^{k+1}} \right) \lambda(Q^{k+1}).
\end{align*}
The first term in the expression on the right-hand side above is non positive, since $d(x) \le 2^k$ and $\overline{d}(Q^{k-2}) \le 2^{k-2}$, so, dividing by $d(x)$, we obtain
\begin{align}
\frac{\mathbb E(L^J)}{d(x)} &= \left(\frac{3}{2d(x)} - \frac{1}{2^{k-1}} \right) \left( 1 - \frac{\overline{d}(Q^{k-1})}{2^{k-1}} \right) \lambda(Q^{k-1}) 
  + \left( \frac{3}{2d(x)} - \frac{1}{2^k} \right) \left( 1 - \frac{\overline{d}(Q^k)}{2^k} \right) \lambda(Q^k) \nonumber \\ 
&+ \left( \frac{1}{2^{k-1}} - \frac{1}{d(x)} \right)\left( 1 - \frac{ \overline{d}(Q^{k+1})}{2^{k+1}} \right) \lambda(Q^{k+1}). \label{eq:EL/d}
\end{align}
Each of the three terms on the right-hand side of~(\ref{eq:EL/d}) is non-negative for $2^{k-1} \le d(x) \le 3 \cdot 2^{k-2}$, and in this case it follows from Lemma~\ref{lem:rho-bound} that
\begin{align*}
\frac{\mathbb E(L^J)}{d(x) \rho} &= \left(\frac{3}{2d(x)} - \frac{1}{2^{k-1}} \right) 2^{k-2} 
  + \left( \frac{3}{2d(x)} - \frac{1}{2^k} \right)2^{k-1} + \left( \frac{1}{2^{k-1}} - \frac{1}{d(x)} \right) 2^k \\
  & = 1 + \frac{2^{k-3}}{d(x)} \\
  & \le 5/4,
\end{align*}
where the maximum is attained at $d(x)=2^{k-1}$. If, on the other hand, $3 \cdot 2^{k-2} \le d(x) \le 2^k$, then the first term on the right-hand side of~(\ref{eq:EL/d}) is negative and the other two are positive. Hence, applying Lemma~\ref{lem:rho-bound} again we obtain
\begin{align*}
\frac{\mathbb E(L^J)}{d(x) \rho} &=    \left( \frac{3}{2d(x)} - \frac{1}{2^k} \right)2^{k-1} + \left( \frac{1}{2^{k-1}} - \frac{1}{d(x)} \right) 2^k \\
  & = 3/2  - 2^{k-2}/d(x)\\
  & \le 5/4,
\end{align*}
where the maximum is attained at $d(x)=2^k$. This completes the proof.
\end{proof}

\section{Randomized pathwise competitive ratio}
\label{sec:pathwise}

In this section we study search strategies for a network $Q$, in the pathwise search paradigm.
We begin by showing that a strategy based on the approach of~\cite{koutsoupias:fixed} is a $3+2\sqrt{2}\approx 5.828$-approximation
of the randomized competitive ratio. Let $r>1$ be a parameter that will be determined 
later, and recall that $Q[r^i] \subseteq Q$ denotes the network of points in $Q$ at distance at most $r^i$ from $O$ for all integers $i$. 
The strategy works in iterations, until
the Hider is located. Namely in iteration $i$, the Searcher follows a {\em Random Chinese Postman Tour} of the network $Q[r^i]$. More precisely, the Searcher
computes a Chinese Postman Tour of $Q[r^i]$, (i.e., a minimal time tour that visits the points of all arcs in $Q[r^i]$), 
then mixes equiprobably between the tour itself and its reversed one. 
Let $s_r$ denote this strategy. The following theorem is based on the approach of~\cite{koutsoupias:fixed}. 

\begin{theorem}
Strategy $s_r$ is a $2(\frac{r}{r-1}+\frac{r}{2})$-approximation of the randomized competitive ratio. In particular, for $r=1+\sqrt{2}$, 
we have that $\rho(s_r) \leq (3+2\sqrt{2})\rho$.
\label{thm:koutsoupias}
\end{theorem}
\begin{proof}
For a fixed Hider strategy, let $j$ denote the iteration in which $s_r$ locates the Hider, and let $C_i$ denote the contribution to the expected search cost of $s_r$
in iteration $i$ for $i \le j$. Moreover, let $l(Q[r^i])$ denote the length of the optimal Chinese Postman Tour in $Q[r^i]$. Then,
\begin{equation}
\rho(s_r)\leq \sup_{j\geq 1} \frac{\sum_{i=-\infty}^j C_i}{r^{j-1}}.
\label{eq:pathwise.koutsoupias.1}
\end{equation}
By considering a Hider that hides uniformly at random on $Q[r^i]$, we obtain that 
\begin{equation}
\rho\geq \frac{l(Q[r^i])}{2r^i}.
\label{eq:pathwise.koutsoupias.2}
\end{equation}  
Lastly, since the Searcher discovers the Hider on iteration $j$, we have that 
\begin{equation}
C_i=l(Q[r^i]), \textrm{ if $i<j$ \ and } \ C_i=\frac{l(Q[r^i])}{2}.
\label{eq:pathwise.koutsoupias.3}
\end{equation}  
By combining~\eqref{eq:pathwise.koutsoupias.1},~\eqref{eq:pathwise.koutsoupias.2},~\eqref{eq:pathwise.koutsoupias.3}, we obtain that
\[
\rho(s_r) \leq 2 \cdot \sup_{j \geq 1} \frac{\sum_{i=-\infty}^{j-1}r^i+\frac{1}{2}r^j}{r^{j-1}} \rho = 2 \left(\frac{r}{r-1}+\frac{r}{2}\right) \rho.
\]
The optimal choice of $r$ that minimizes the above expression is $r=1+\sqrt{2}$, from which it follows that 
$\rho(s_r) \leq (3+2\sqrt{2})\rho$.
\end{proof}
We now show how the randomized doubling strategy of Section~\ref{subsec:randomized.5/4} can be adapted for the pathwise case to give an improved approximation ratio of $5$. We define the shortest path trees $Q_j$ and the random levels, $R_j, j \in \mathbb{Z}$ as in the expanding search setting. A Randomized Depth-First Search (RDFS) is defined similarly in the pathwise search setting as an equiprobable mixture of a depth-first search $S$ and its reverse search $S^{-1}$, except that we stipulate that $S$ and $S^{-1}$ return to the root $O$ after visiting all the leaf nodes. We will use the following lemma, whose proof can be found, for example, in~\cite{searchgames}.
\begin{lemma} \label{lem:pw-RDFS}
	Let $s$ be a (pathwise) RDFS of a tree $\mathcal{T}$. Then the expected time $T(s,H)$ a point $H \in \mathcal{T}$ is found by $s$ satisfies
	\[
	T(s,H) \le \lambda(\mathcal{T}).
	\]
\end{lemma}
The (pathwise) random doubling strategy then performs successive RDFS's of {\em unions} of levels $\cup_{i \le j} R_j$.
\begin{theorem}
The (pathwise) random doubling strategy $s$ is a $5$-approximation for the optimal randomized pathwise search. That is, $\rho_s \le 5 \rho$.
\end{theorem}
\begin{proof}
As in the proof of Theorem~\ref{thm:5/4}, suppose that the randomized competitive ratio of $s$ is maximized at some point $x$ which is contained in $Q_k$, for some $k$. Again, we define $J$ as the index of the level containing $x$. We use Lemma~\ref{lem:pw-RDFS} to write down an expression for the expected search time of $x$, conditioned on $J$, and which we will denote by $T(s,x|J)$.
\[
T(s,x|J) = \lambda(\cup_{i \le J} R_i) + \sum_{j \le J-1} 2 \lambda(\cup_{i \le j} R_i).
\]
Rearranging, we have 
\begin{align*}
T(s,x|J) &= \lambda(R_J) + \lambda(\cup_{i \le J-1} R_i) + \sum_{j \le J-1} \lambda(R_j) + \lambda(\cup_{i \le j-1} R_i) + \lambda(\cup_{i \le j} R_i)  \\
& = \left( \lambda(R_J)  + \sum_{j \le J-1}  \lambda(R_j) \right) + \left( \lambda(\cup_{i \le J-1} R_i)+ \sum_{j \le J-2} \lambda(\cup_{i \le j} R_i) \right) + \sum_{j \le J-1} \lambda(\cup_{i \le j} R_i) \\
& = \sum_{j \le J} \lambda(R_i) + 2 \lambda(\cup_{i \le j-1} R_i).
\end{align*}
Now taking expectations, with respect to $J$, we obtain
\begin{align}
T(s,x)  = \sum_{j \le J} 2 \mathbb{E} (L^j), \label{eq:pw-ratio}
\end{align}
where $L^J = \lambda(R_J)/2 + \lambda(\cup_{i \le J-1} R_i)$ is defined as in the proof of Theorem~\ref{thm:5/4}, and similarly for $L^{J-1},L^{J-2},$ etc. We showed in the proof of Theorem~\ref{thm:5/4} that
\[
\frac{\mathbb{E}(L^J)}{d(x) \rho} \le 5/4,
\]
and it follows that
\[
\frac{\mathbb{E}(L^{J-j})}{ 2^{-j} d(x) \rho} \le 5/4.
\]
So by~(\ref{eq:pw-ratio}),
\begin{align*}
\frac{\hat{T}(s,x)}{ \rho} &\le \sum_{j \ge 0} 2 \cdot 2^{-j} \frac{\mathbb{E}(L^{J-j})}{2^{-j}d(x) \rho} \\
& \le \sum_{j \ge 0} 2 \cdot 2^{-j} \cdot 5/4 \\
&=5.
\end{align*}
\end{proof}

\section{Implementation and complexity issues}
\label{sec:implementation}

In this section we discuss issues related to the implementation of our search strategies. 

\paragraph{Infinitesimally small tours}
For the purposes of the analysis, we allow the doubling strategies of Sections~\ref{subsec:randomized.better} and~\ref{sec:pathwise} to start with an infinite number of infinitesimally small tours. This is a standard way of getting around the technical complication that any strategy which starts by a search to a constant distance $c>0$ from the root cannot be constant-competitive, and has been applied in the analysis of searching in the infinite line and the infinite star, e.g.,~\cite{gal:minimax}. In practice, of course, the Searcher will start its search at some small distance from the root, and we may assume, as often in the Computer Science literature on search algorithms, that the Hider is at distance at least 1 from the origin. The overall analysis remains the same, at the expense of some negligible additive contribution to the overall search cost that does not affect the competitive ratios. 

\paragraph{Representation of the network}
If the network $Q$ is bounded, then there is a straightforward way of representing it as an undirected, weighted graph, in which the edge weights correspond to arc lengths. However, if $Q$ is unbounded, we need certain assumptions in regards to how the Searcher can access the network. One such way is to assume an oracle that given a parameter $r>0$ returns the subnetwork $Q[r]$ of $Q$ that corresponds to all points in $Q$ at a radius $r$ around the root, and which in turn can be encoded as a weighted graph, since it is bounded. For a Hider $H$, and a given search strategy, we denote by $O_H$ the number
of accesses to the oracle that the strategy requires (and which we aim to bound).

\paragraph{}
With these observations in mind, we can now discuss the implementation of our strategies. Concerning the ``waterfilling'' deterministic strategy of Section~\ref{sec:deterministic}, it suffices for the Searcher to access the oracle a logarithmic number of times, namely $O_H=O(\log(d(H)))$. Specifically, the oracle will reveal the subnetworks $Q[2^i]$, with 
$i \in [1, \lceil \log(d(H)) \rceil]$. Within any given subnetwork, the strategy can be implemented in time polynomial in its
graph representation, by simply keeping track of the ``active'' edges, namely arcs of the network which have been only partially searched.

Similarly, for the doubling strategies of Sections~\ref{subsec:randomized.better} and~\ref{sec:pathwise}, a logarithimic number of oracle accesses, in the distance of the Hider, will suffice. For a given level $j$ in the execution of these algorithms, all associated actions of the strategies, namely finding a shortest path tree, performing an RDFS traversal of the tree, or finding a Chinese Postman tour can be done in time polynomial in the size of the graph representation of the corresponding level.

\section{Conclusion}
\label{sec:conclusion}

In this work we studied expanding and pathwise search in a general, possibly unbounded network. We focused on the competitive ratio
of the network as a measure for the efficiency of a search strategy, and gave the first constant-approximation mixed strategies in
these settings. In particular, we addressed two open questions from~\cite{ANGELOPOULOS201951}, namely how to derive efficient
strategies that are i) randomized; and ii) apply to a general network and not only to discrete trees. 

The obvious open problem from our work is to further improve the approximation of the randomized competitive ratios, or identify
more classes of networks for which optimal strategies can be found (although our result of Section~\ref{subsec:randomized.y} 
shows that any such  identification will unavoidably exclude some very simple networks). Another direction is to consider searching
for multiple hiders, as an extension to multi-hider search in a star under the competitive ratio~\cite{multi-target} or relaxations
of the competitive ratio~\cite{oil,hyperbolic}.

Last, concerning bounded networks, and beyond competitive analysis of search strategies, an interesting, and perhaps surprisingly open 
problem is to find a pathwise search strategy that minimizes the expected time to locate the Hider, assuming that the Hider's distribution is known. 

\paragraph{Acknowledgements}
This research benefited from the support of the FMJH Program Gaspard Monge for optimization and operations research and their interactions with data science, and from the support from EDF, Thales and Orange.

\bibliographystyle{apacite}
\bibliography{targets}

\begin{thebibliography}{}

\bibitem [\protect \citeauthoryear {%
Alpern%
\ \BBA {} Gal%
}{%
Alpern%
\ \BBA {} Gal%
}{%
{\protect \APACyear {1988}}%
}]{%
AG:minimax}
\APACinsertmetastar {%
AG:minimax}%
\begin{APACrefauthors}%
Alpern, S.%
\BCBT {}\ \BBA {} Gal, S.%
\end{APACrefauthors}%
\unskip\
\newblock
\APACrefYearMonthDay{1988}{}{}.
\newblock
{\BBOQ}\APACrefatitle {A mixed strategy minimax theorem without compactness} {A
  mixed strategy minimax theorem without compactness}.{\BBCQ}
\newblock
\APACjournalVolNumPages{SIAM Journal on Control and
  Optimization}{26}{6}{1357-1361}.
\PrintBackRefs{\CurrentBib}

\bibitem [\protect \citeauthoryear {%
Alpern%
\ \BBA {} Gal%
}{%
Alpern%
\ \BBA {} Gal%
}{%
{\protect \APACyear {2003}}%
}]{%
searchgames}
\APACinsertmetastar {%
searchgames}%
\begin{APACrefauthors}%
Alpern, S.%
\BCBT {}\ \BBA {} Gal, S.%
\end{APACrefauthors}%
\unskip\
\newblock
\APACrefYear{2003}.
\newblock
\APACrefbtitle {The theory of search games and rendezvous} {The theory of
  search games and rendezvous}.
\newblock
\APACaddressPublisher{}{Kluwer Academic Publishers}.
\PrintBackRefs{\CurrentBib}

\bibitem [\protect \citeauthoryear {%
Alpern%
\ \BBA {} Lidbetter%
}{%
Alpern%
\ \BBA {} Lidbetter%
}{%
{\protect \APACyear {2013}}%
}]{%
AL:expanding}
\APACinsertmetastar {%
AL:expanding}%
\begin{APACrefauthors}%
Alpern, S.%
\BCBT {}\ \BBA {} Lidbetter, T.%
\end{APACrefauthors}%
\unskip\
\newblock
\APACrefYearMonthDay{2013}{}{}.
\newblock
{\BBOQ}\APACrefatitle {Mining Coal or Finding Terrorists: The Expanding Search
  Paradigm} {Mining coal or finding terrorists: The expanding search
  paradigm}.{\BBCQ}
\newblock
\APACjournalVolNumPages{Operations Research}{61}{2}{265--279}.
\PrintBackRefs{\CurrentBib}

\bibitem [\protect \citeauthoryear {%
Alpern%
\ \BBA {} Lidbetter%
}{%
Alpern%
\ \BBA {} Lidbetter%
}{%
{\protect \APACyear {2019}}%
}]{%
AL:19.general}
\APACinsertmetastar {%
AL:19.general}%
\begin{APACrefauthors}%
Alpern, S.%
\BCBT {}\ \BBA {} Lidbetter, T.%
\end{APACrefauthors}%
\unskip\
\newblock
\APACrefYearMonthDay{2019}{}{}.
\newblock
{\BBOQ}\APACrefatitle {Approximate solutions for expanding search games on
  general networks} {Approximate solutions for expanding search games on
  general networks}.{\BBCQ}
\newblock
\APACjournalVolNumPages{Annals of Operations Research}{275}{}{259--279}.
\PrintBackRefs{\CurrentBib}

\bibitem [\protect \citeauthoryear {%
Angelopoulos%
, %
, Ars{\'e}nio%
, D{\"u}rr%
\BCBL {}\ \BBA {} L{\'o}pez-Ortiz%
}{%
Angelopoulos%
, %
\BCBL {}\ \protect \BOthers {.}}{%
{\protect \APACyear {2016}}%
}]{%
Angelopoulos2016}
\APACinsertmetastar {%
Angelopoulos2016}%
\begin{APACrefauthors}%
Angelopoulos, S.%
, %
, Ars{\'e}nio, D.%
, D{\"u}rr, C.%
\BCBL {}\ \BBA {} L{\'o}pez-Ortiz, A.%
\end{APACrefauthors}%
\unskip\
\newblock
\APACrefYearMonthDay{2016}{}{}.
\newblock
{\BBOQ}\APACrefatitle {Multi-processor Search and Scheduling Problems with
  Setup Cost} {Multi-processor search and scheduling problems with setup
  cost}.{\BBCQ}
\newblock
\APACjournalVolNumPages{Theory of Computing Systems}{}{}{1--34}.
\PrintBackRefs{\CurrentBib}

\bibitem [\protect \citeauthoryear {%
Angelopoulos%
, Ars{\'e}nio%
\BCBL {}\ \BBA {} D{\"u}rr%
}{%
Angelopoulos%
\ \protect \BOthers {.}}{%
{\protect \APACyear {2017}}%
}]{%
Angelopoulos2017}
\APACinsertmetastar {%
Angelopoulos2017}%
\begin{APACrefauthors}%
Angelopoulos, S.%
, Ars{\'e}nio, D.%
\BCBL {}\ \BBA {} D{\"u}rr, C.%
\end{APACrefauthors}%
\unskip\
\newblock
\APACrefYearMonthDay{2017}{}{}.
\newblock
{\BBOQ}\APACrefatitle {Infinite linear programming and online searching with
  turn cost} {Infinite linear programming and online searching with turn
  cost}.{\BBCQ}
\newblock
\APACjournalVolNumPages{Theoretical Computer Science}{670}{}{11--22}.
\PrintBackRefs{\CurrentBib}

\bibitem [\protect \citeauthoryear {%
Angelopoulos%
, D\"urr%
\BCBL {}\ \BBA {} Lidbetter%
}{%
Angelopoulos%
, D\"urr%
\BCBL {}\ \BBA {} Lidbetter%
}{%
{\protect \APACyear {2016}}%
}]{%
stacs-expanding}
\APACinsertmetastar {%
stacs-expanding}%
\begin{APACrefauthors}%
Angelopoulos, S.%
, D\"urr, C.%
\BCBL {}\ \BBA {} Lidbetter, T.%
\end{APACrefauthors}%
\unskip\
\newblock
\APACrefYearMonthDay{2016}{}{}.
\newblock
{\BBOQ}\APACrefatitle {The expanding search ratio of a graph} {The expanding
  search ratio of a graph}.{\BBCQ}
\newblock
\BIn{} \APACrefbtitle {Proceedings of the 33rd {International} {Symposium} on
  {Theoretical} {Aspects} of {Computer} {Science} (STACS)} {Proceedings of the
  33rd {International} {Symposium} on {Theoretical} {Aspects} of {Computer}
  {Science} (stacs)}\ (\BPGS\ 9:1--9:14).
\PrintBackRefs{\CurrentBib}

\bibitem [\protect \citeauthoryear {%
Angelopoulos%
, D{\" u}rr%
\BCBL {}\ \BBA {} Lidbetter%
}{%
Angelopoulos%
\ \protect \BOthers {.}}{%
{\protect \APACyear {2019}}%
}]{%
ANGELOPOULOS201951}
\APACinsertmetastar {%
ANGELOPOULOS201951}%
\begin{APACrefauthors}%
Angelopoulos, S.%
, D{\" u}rr, C.%
\BCBL {}\ \BBA {} Lidbetter, T.%
\end{APACrefauthors}%
\unskip\
\newblock
\APACrefYearMonthDay{2019}{}{}.
\newblock
{\BBOQ}\APACrefatitle {The expanding search ratio of a graph} {The expanding
  search ratio of a graph}.{\BBCQ}
\newblock
\APACjournalVolNumPages{Discrete Applied Mathematics}{260}{}{51--65}.
\PrintBackRefs{\CurrentBib}

\bibitem [\protect \citeauthoryear {%
Angelopoulos%
, L\'opez-Ortiz%
\BCBL {}\ \BBA {} Panagiotou%
}{%
Angelopoulos%
\ \protect \BOthers {.}}{%
{\protect \APACyear {2014}}%
}]{%
multi-target}
\APACinsertmetastar {%
multi-target}%
\begin{APACrefauthors}%
Angelopoulos, S.%
, L\'opez-Ortiz, A.%
\BCBL {}\ \BBA {} Panagiotou, K.%
\end{APACrefauthors}%
\unskip\
\newblock
\APACrefYearMonthDay{2014}{}{}.
\newblock
{\BBOQ}\APACrefatitle {Multi-target ray searching problems} {Multi-target ray
  searching problems}.{\BBCQ}
\newblock
\APACjournalVolNumPages{Theoretical Computer Science}{540}{}{2--12}.
\PrintBackRefs{\CurrentBib}

\bibitem [\protect \citeauthoryear {%
Ausiello%
, Leonardi%
\BCBL {}\ \BBA {} Marchetti{-}Spaccamela%
}{%
Ausiello%
\ \protect \BOthers {.}}{%
{\protect \APACyear {2000}}%
}]{%
DBLP:conf/ciac/AusielloLM00}
\APACinsertmetastar {%
DBLP:conf/ciac/AusielloLM00}%
\begin{APACrefauthors}%
Ausiello, G.%
, Leonardi, S.%
\BCBL {}\ \BBA {} Marchetti{-}Spaccamela, A.%
\end{APACrefauthors}%
\unskip\
\newblock
\APACrefYearMonthDay{2000}{}{}.
\newblock
{\BBOQ}\APACrefatitle {On Salesmen, Repairmen, Spiders, and Other Traveling
  Agents} {On salesmen, repairmen, spiders, and other traveling agents}.{\BBCQ}
\newblock
\BIn{} \APACrefbtitle {Proceedings of 4th {Italian} {Conference} on
  {Algorithms} and {Complexity}, {CIAC}} {Proceedings of 4th {Italian}
  {Conference} on {Algorithms} and {Complexity}, {CIAC}}\ (\BPGS\ 1--16).
\PrintBackRefs{\CurrentBib}

\bibitem [\protect \citeauthoryear {%
Beck%
\ \BBA {} Newman%
}{%
Beck%
\ \BBA {} Newman%
}{%
{\protect \APACyear {1970}}%
}]{%
beck:yet.more}
\APACinsertmetastar {%
beck:yet.more}%
\begin{APACrefauthors}%
Beck, A.%
\BCBT {}\ \BBA {} Newman, D.%
\end{APACrefauthors}%
\unskip\
\newblock
\APACrefYearMonthDay{1970}{}{}.
\newblock
{\BBOQ}\APACrefatitle {Yet more on the linear search problem} {Yet more on the
  linear search problem}.{\BBCQ}
\newblock
\APACjournalVolNumPages{Israel Journal of Mathematics}{8}{}{419--429}.
\PrintBackRefs{\CurrentBib}

\bibitem [\protect \citeauthoryear {%
Bose%
, Carufel%
\BCBL {}\ \BBA {} Durocher%
}{%
Bose%
\ \protect \BOthers {.}}{%
{\protect \APACyear {2015}}%
}]{%
revisiting:esa}
\APACinsertmetastar {%
revisiting:esa}%
\begin{APACrefauthors}%
Bose, P.%
, Carufel, J\BPBI D.%
\BCBL {}\ \BBA {} Durocher, S.%
\end{APACrefauthors}%
\unskip\
\newblock
\APACrefYearMonthDay{2015}{}{}.
\newblock
{\BBOQ}\APACrefatitle {Searching on a line: A complete characterization of the
  optimal solution.} {Searching on a line: A complete characterization of the
  optimal solution.}{\BBCQ}
\newblock
\APACjournalVolNumPages{Theoretical Computer Science}{569}{}{24--42}.
\PrintBackRefs{\CurrentBib}

\bibitem [\protect \citeauthoryear {%
Demaine%
, Fekete%
\BCBL {}\ \BBA {} Gal%
}{%
Demaine%
\ \protect \BOthers {.}}{%
{\protect \APACyear {2006}}%
}]{%
demaine:turn}
\APACinsertmetastar {%
demaine:turn}%
\begin{APACrefauthors}%
Demaine, E.%
, Fekete, S.%
\BCBL {}\ \BBA {} Gal, S.%
\end{APACrefauthors}%
\unskip\
\newblock
\APACrefYearMonthDay{2006}{}{}.
\newblock
{\BBOQ}\APACrefatitle {Online searching with turn cost} {Online searching with
  turn cost}.{\BBCQ}
\newblock
\APACjournalVolNumPages{Theoretical Computer Science}{361}{}{342-355}.
\PrintBackRefs{\CurrentBib}

\bibitem [\protect \citeauthoryear {%
Fleischer%
, Kamphans%
, Klein%
, Langetepe%
\BCBL {}\ \BBA {} Trippen%
}{%
Fleischer%
\ \protect \BOthers {.}}{%
{\protect \APACyear {2008}}%
}]{%
fleischer:online}
\APACinsertmetastar {%
fleischer:online}%
\begin{APACrefauthors}%
Fleischer, R.%
, Kamphans, T.%
, Klein, R.%
, Langetepe, E.%
\BCBL {}\ \BBA {} Trippen, G.%
\end{APACrefauthors}%
\unskip\
\newblock
\APACrefYearMonthDay{2008}{}{}.
\newblock
{\BBOQ}\APACrefatitle {Competitive online approximation of the optimal search
  ratio} {Competitive online approximation of the optimal search ratio}.{\BBCQ}
\newblock
\APACjournalVolNumPages{SIAM Journal on Computing}{38}{3}{881-898}.
\PrintBackRefs{\CurrentBib}

\bibitem [\protect \citeauthoryear {%
Gal%
}{%
Gal%
}{%
{\protect \APACyear {1972}}%
}]{%
gal:general}
\APACinsertmetastar {%
gal:general}%
\begin{APACrefauthors}%
Gal, S.%
\end{APACrefauthors}%
\unskip\
\newblock
\APACrefYearMonthDay{1972}{}{}.
\newblock
{\BBOQ}\APACrefatitle {A general search game} {A general search game}.{\BBCQ}
\newblock
\APACjournalVolNumPages{Israel Journal of Mathematics}{12}{}{32--45}.
\PrintBackRefs{\CurrentBib}

\bibitem [\protect \citeauthoryear {%
Gal%
}{%
Gal%
}{%
{\protect \APACyear {1974}}%
}]{%
gal:minimax}
\APACinsertmetastar {%
gal:minimax}%
\begin{APACrefauthors}%
Gal, S.%
\end{APACrefauthors}%
\unskip\
\newblock
\APACrefYearMonthDay{1974}{}{}.
\newblock
{\BBOQ}\APACrefatitle {Minimax solutions for linear search problems} {Minimax
  solutions for linear search problems}.{\BBCQ}
\newblock
\APACjournalVolNumPages{SIAM Journal on Applied Mathematics}{27}{}{17--30}.
\PrintBackRefs{\CurrentBib}

\bibitem [\protect \citeauthoryear {%
Gal%
}{%
Gal%
}{%
{\protect \APACyear {1979}}%
}]{%
gal1979search}
\APACinsertmetastar {%
gal1979search}%
\begin{APACrefauthors}%
Gal, S.%
\end{APACrefauthors}%
\unskip\
\newblock
\APACrefYearMonthDay{1979}{}{}.
\newblock
{\BBOQ}\APACrefatitle {Search games with mobile and immobile hider} {Search
  games with mobile and immobile hider}.{\BBCQ}
\newblock
\APACjournalVolNumPages{SIAM Journal on Control and
  Optimization}{17}{1}{99--122}.
\PrintBackRefs{\CurrentBib}

\bibitem [\protect \citeauthoryear {%
Hipke%
, Icking%
, Klein%
\BCBL {}\ \BBA {} Langetepe%
}{%
Hipke%
\ \protect \BOthers {.}}{%
{\protect \APACyear {1999}}%
}]{%
DBLP:journals/dam/HipkeIKL99}
\APACinsertmetastar {%
DBLP:journals/dam/HipkeIKL99}%
\begin{APACrefauthors}%
Hipke, C\BPBI A.%
, Icking, C.%
, Klein, R.%
\BCBL {}\ \BBA {} Langetepe, E.%
\end{APACrefauthors}%
\unskip\
\newblock
\APACrefYearMonthDay{1999}{}{}.
\newblock
{\BBOQ}\APACrefatitle {How to Find a Point on a Line Within a Fixed Distance}
  {How to find a point on a line within a fixed distance}.{\BBCQ}
\newblock
\APACjournalVolNumPages{Discrete Applied Mathematics}{93}{1}{67--73}.
\PrintBackRefs{\CurrentBib}

\bibitem [\protect \citeauthoryear {%
Isaacs%
}{%
Isaacs%
}{%
{\protect \APACyear {1965}}%
}]{%
isaacs1999differential}
\APACinsertmetastar {%
isaacs1999differential}%
\begin{APACrefauthors}%
Isaacs, R.%
\end{APACrefauthors}%
\unskip\
\newblock
\APACrefYear{1965}.
\newblock
\APACrefbtitle {Differential games} {Differential games}.
\newblock
\APACaddressPublisher{}{John Wiley and Sons, New York}.
\PrintBackRefs{\CurrentBib}

\bibitem [\protect \citeauthoryear {%
Jaillet%
\ \BBA {} Stafford%
}{%
Jaillet%
\ \BBA {} Stafford%
}{%
{\protect \APACyear {1993}}%
}]{%
jaillet:online}
\APACinsertmetastar {%
jaillet:online}%
\begin{APACrefauthors}%
Jaillet, P.%
\BCBT {}\ \BBA {} Stafford, M.%
\end{APACrefauthors}%
\unskip\
\newblock
\APACrefYearMonthDay{1993}{}{}.
\newblock
{\BBOQ}\APACrefatitle {Online searching} {Online searching}.{\BBCQ}
\newblock
\APACjournalVolNumPages{Operations Research}{49}{}{234--244}.
\PrintBackRefs{\CurrentBib}

\bibitem [\protect \citeauthoryear {%
Kao%
, Ma%
, Sipser%
\BCBL {}\ \BBA {} Yin%
}{%
Kao%
\ \protect \BOthers {.}}{%
{\protect \APACyear {1998}}%
}]{%
hybrid}
\APACinsertmetastar {%
hybrid}%
\begin{APACrefauthors}%
Kao, M\BHBI Y.%
, Ma, Y.%
, Sipser, M.%
\BCBL {}\ \BBA {} Yin, Y.%
\end{APACrefauthors}%
\unskip\
\newblock
\APACrefYearMonthDay{1998}{}{}.
\newblock
{\BBOQ}\APACrefatitle {Optimal constructions of hybrid algorithms} {Optimal
  constructions of hybrid algorithms}.{\BBCQ}
\newblock
\APACjournalVolNumPages{Journal of Algorithms}{29}{1}{142--164}.
\PrintBackRefs{\CurrentBib}

\bibitem [\protect \citeauthoryear {%
Kirkpatrick%
}{%
Kirkpatrick%
}{%
{\protect \APACyear {2009}}%
}]{%
hyperbolic}
\APACinsertmetastar {%
hyperbolic}%
\begin{APACrefauthors}%
Kirkpatrick, D\BPBI G.%
\end{APACrefauthors}%
\unskip\
\newblock
\APACrefYearMonthDay{2009}{}{}.
\newblock
{\BBOQ}\APACrefatitle {Hyperbolic dovetailing} {Hyperbolic dovetailing}.{\BBCQ}
\newblock
\BIn{} \APACrefbtitle {Proceedings of the 17th {European} {Symposium} on
  {Algorithms} ({ESA})} {Proceedings of the 17th {European} {Symposium} on
  {Algorithms} ({ESA})}\ (\BPGS\ 616--627).
\PrintBackRefs{\CurrentBib}

\bibitem [\protect \citeauthoryear {%
Koutsoupias%
, Papadimitriou%
\BCBL {}\ \BBA {} Yannakakis%
}{%
Koutsoupias%
\ \protect \BOthers {.}}{%
{\protect \APACyear {1996}}%
}]{%
koutsoupias:fixed}
\APACinsertmetastar {%
koutsoupias:fixed}%
\begin{APACrefauthors}%
Koutsoupias, E.%
, Papadimitriou, C.%
\BCBL {}\ \BBA {} Yannakakis, M.%
\end{APACrefauthors}%
\unskip\
\newblock
\APACrefYearMonthDay{1996}{}{}.
\newblock
{\BBOQ}\APACrefatitle {Searching a fixed graph} {Searching a fixed
  graph}.{\BBCQ}
\newblock
\BIn{} \APACrefbtitle {Proceedings of the 23rd {International} {Colloquium} on
  {Automata}, {Languages} and {Programming} ({ICALP})} {Proceedings of the 23rd
  {International} {Colloquium} on {Automata}, {Languages} and {Programming}
  ({ICALP})}\ (\BPG~280-289).
\PrintBackRefs{\CurrentBib}

\bibitem [\protect \citeauthoryear {%
L\'opez-Ortiz%
\ \BBA {} Schuierer%
}{%
L\'opez-Ortiz%
\ \BBA {} Schuierer%
}{%
{\protect \APACyear {2001}}%
}]{%
ultimate}
\APACinsertmetastar {%
ultimate}%
\begin{APACrefauthors}%
L\'opez-Ortiz, A.%
\BCBT {}\ \BBA {} Schuierer, S.%
\end{APACrefauthors}%
\unskip\
\newblock
\APACrefYearMonthDay{2001}{}{}.
\newblock
{\BBOQ}\APACrefatitle {The ultimate strategy to search on $m$ rays?} {The
  ultimate strategy to search on $m$ rays?}{\BBCQ}
\newblock
\APACjournalVolNumPages{Theoretical Computer Science}{261}{2}{267--295}.
\PrintBackRefs{\CurrentBib}

\bibitem [\protect \citeauthoryear {%
L\'opez-Ortiz%
\ \BBA {} Schuierer%
}{%
L\'opez-Ortiz%
\ \BBA {} Schuierer%
}{%
{\protect \APACyear {2004}}%
}]{%
alex:robots}
\APACinsertmetastar {%
alex:robots}%
\begin{APACrefauthors}%
L\'opez-Ortiz, A.%
\BCBT {}\ \BBA {} Schuierer, S.%
\end{APACrefauthors}%
\unskip\
\newblock
\APACrefYearMonthDay{2004}{}{}.
\newblock
{\BBOQ}\APACrefatitle {On-line parallel heuristics, processor scheduling and
  robot searching under the competitive framework} {On-line parallel
  heuristics, processor scheduling and robot searching under the competitive
  framework}.{\BBCQ}
\newblock
\APACjournalVolNumPages{Theoretical Computer Science}{310}{1--3}{527--537}.
\PrintBackRefs{\CurrentBib}

\bibitem [\protect \citeauthoryear {%
McGregor%
, Onak%
\BCBL {}\ \BBA {} Panigrahy%
}{%
McGregor%
\ \protect \BOthers {.}}{%
{\protect \APACyear {2009}}%
}]{%
oil}
\APACinsertmetastar {%
oil}%
\begin{APACrefauthors}%
McGregor, A.%
, Onak, K.%
\BCBL {}\ \BBA {} Panigrahy, R.%
\end{APACrefauthors}%
\unskip\
\newblock
\APACrefYearMonthDay{2009}{}{}.
\newblock
{\BBOQ}\APACrefatitle {The oil searching problem} {The oil searching
  problem}.{\BBCQ}
\newblock
\BIn{} \APACrefbtitle {Proceedings of the 17th {European} {Symposium} on
  {Algorithms} ({ESA})} {Proceedings of the 17th {European} {Symposium} on
  {Algorithms} ({ESA})}\ (\BPGS\ 504--515).
\PrintBackRefs{\CurrentBib}

\end{thebibliography}

\end{document}